\documentclass[11pt,twoside,reqno]{amsart}
\allowdisplaybreaks
\usepackage{amsmath,amstext,amssymb,epsfig,multicol,enumerate}
\usepackage{graphicx}
\usepackage{mathrsfs}
\calclayout
\makeatletter
\renewcommand{\@seccntformat}[1]{\bf\csname the#1\endcsname.}
\renewcommand{\section}{\@startsection{section}{1}
	\z@{.7\linespacing\@plus\linespacing}{.5\linespacing}
	{\normalfont\upshape\bfseries\centering}}
\renewcommand{\@biblabel}[1]{\@ifnotempty{#1}{#1.}}
\makeatother
\theoremstyle{plain}
\newtheorem{thm}{Theorem}[section]
\newtheorem{lem}[thm]{Lemma}
\newtheorem{prop}[thm]{Proposition}
\newtheorem{cor}[thm]{Corollary}

\theoremstyle{definition}

\newtheorem{defn}[thm]{Definition}

\usepackage{cancel}
\usepackage[parfill]{parskip}
\usepackage[german]{varioref}
\usepackage[all]{xy}
\usepackage{color}

\def \>{\succ}
\def \<{\prec}

\begin{document}	
\title[ Bouzid Mosbahi\textsuperscript{1}, Ahmed Zahari\textsuperscript{2} ]{An Algorithmic Approach to Inner Derivations of Low-Dimensional Zinbiel Algebras}
	\author{Bouzid Mosbahi\textsuperscript{1}, Ahmed Zahari\textsuperscript{2}}
\address{\textsuperscript{}Department of Mathematics, Faculty of Sciences, University of Sfax, Sfax, Tunisia}
 \address{\textsuperscript{2}
IRIMAS-Department of Mathematics, Faculty of Sciences, University of Haute Alsace, Mulhouse, France}
\email{\textsuperscript{1}mosbahi.bouzid.etud@fss.usf.tn}
\email{\textsuperscript{2}abdou-damdji.ahmed-zahari@uha.fr}

	\keywords{Zinbiel Algebra, derivation, inner derivation}
	\subjclass[2020]{16D70,17A30, 17A32}

	\date{\today}

	\begin{abstract}
In this paper, we introduce the concept of inner derivations of low-dimensional Zinbiel algebras and investigate their properties. The primary objective of this study is to develop an algorithm to characterize the inner derivations of any n-dimensional Zinbiel algebra in matrix form. Additionally, we apply this algorithm to two, three and four-dimensional complex Zinbiel algebras, providing explicit descriptions of their inner derivations.
\end{abstract}

\maketitle \section{ Introduction}\label{introduction}
Zinbiel algebras were introduced by Loday in 1995 \cite{1} while studying the cup product in Leibniz cohomology. These algebras are non-associative structures that are Koszul dual to Leibniz algebras, creating an important connection between their operads \cite{2}. Since their discovery, Zinbiel algebras have been studied widely because of their strong links to other algebraic structures, including commutative algebras, Lie algebras, and operads. There are many theoretical studies on the derivations and inner derivations of associative, Lie, and Leibniz algebras, see \cite{3,4,5,6,7,8}. However, few of them are in the case of Zinbiel algebras and particularly in their computation. For this reason, our interest in this paper is to study and describe the inner derivations of low-dimensional Zinbiel algebras. These algebras play an important role in understanding higher-dimensional structures and their properties. Many works have classified Zinbiel algebras of dimensions up to 4 \cite{9,10,11}, which has helped reveal their structure and paved the way for studying their derivations, automorphisms, and other features.

In this paper, we focus on \textit{inner derivations} of Zinbiel algebras, which are specific types of linear maps determined by the elements of the algebra. Our main goal is to create an algorithm to represent inner derivations in matrix form for any $n$-dimensional Zinbiel algebra. To illustrate this, We apply our algorithm to two-, three-, and four-dimensional complex Zinbiel algebras using Mathematica or Maple software, providing detailed descriptions of their inner derivations. This work contributes to a better understanding of derivations in non-associative algebras and provides a practical method for further research on Zinbiel algebras.

This paper is organized as follows:  
\begin{itemize}
    \item In Section 1, we give an overview of Zinbiel algebras and summarize key results about their structure and derivations.
    \item Section 2 explains the basic concepts and definitions needed for the study, including the idea of inner derivations and their importance.
    \item In Section 3, we explore fundamental properties of inner derivations for low-dimensional Zinbiel algebras to build a foundation for the main results.
    \item Section 4 introduces an algorithm to compute inner derivations for $n$-dimensional Zinbiel algebras. We also demonstrate this algorithm with examples from two-, three-, and four-dimensional complex Zinbiel algebras, based on classification results in \cite{12,13,14,15,16}.
    \item Finally, the conclusion summarizes our findings and suggests future directions for research.
\end{itemize}

\section{ Prelimieries}
\begin{defn}
A \textit{Zinbiel algebra} $A$ is a vector space over a field $\mathbb{C}$ equipped with a bilinear map\\ $\circ : A \times A \to A$, satisfying the identity:  
\begin{align*}
(u \circ v) \circ w &= u \circ (v \circ w) + u \circ (w \circ v), \quad \forall u, v, w \in A.
\end{align*}

In a Zinbiel algebra, we consider the \textit{right multiplication operator} \( R_u \) and the \textit{left multiplication operator} \( L_u \), defined as follows:
\begin{align*}
R_u(v) &= v \circ u, \quad \forall v \in A,\\
L_u(v) &= u \circ v, \quad \forall v \in A.
\end{align*}
\end{defn}

\begin{defn}
Let $A_1$  and  $A_2$ be two Zinbiel algebras over a field $\mathbb{C}$.  
A homomorphism between $A_1$ and $A_2$ is a  $\mathbb{C}$-linear mapping  
\begin{align*}
f : A_1 \to A_2
\end{align*}
such that  
\begin{align*}
f(u \circ v) = f(u) \circ f(v), \quad \forall u, v \in A_1,
\end{align*} 
where  $\circ$ denotes the Zinbiel product.

The set of all homomorphisms of  $A$ is denoted by $Hom(A)$.

In what follows, we define a bilinear map  $[\cdot, \cdot] : A \times A \to A$  by  
\begin{align*}
[u, v] &= u \circ v - v \circ u, \quad \forall u, v \in A.
\end{align*}
\end{defn}

\begin{defn}
A derivation of a Zinbiel algebra \(A\) is a linear transformation \(d : A \to A\), where  
\begin{align*}
d(u \circ v) &= d(u) \circ v + u \circ d(v), \quad \forall u, v \in A.
\end{align*} 
We denote the set of all derivations of a Zinbiel algebra \(A\) by \(\mathrm{Der}(A)\). The set \(\mathrm{Der}(A)\) forms an associative algebra with respect to the composition operation \(\circ\), and it is a Lie algebra with respect to the bracket \([d_1, d_2] = d_1 \circ d_2 - d_2 \circ d_1\).  

In the following, we provide some earlier results on the properties of derivations of Zinbiel algebras. 
\end{defn}

\begin{prop}
Let $(A, \circ)$ be a Zinbiel algebra and $d \in Hom(A)$. Then the following conditions are equivalent:
\begin{enumerate}
    \item[(i)] $d \in Der(A)$,
    \item[(ii)] $[L_u, L_d] = L_{d(u)}$ for all $u \in A$, 
    \item[(iii)] $[R_u, R_d] = R_{d(u)}$ for all $u \in A$,
\end{enumerate}
where $L_u(v) = u \circ v$ and $R_u(v) = v  \circ u$ are the left and right multiplication operators, respectively.
\end{prop}

\begin{proof}
\textbf{(i) $\implies$ (ii):}  
Assume $d \in \operatorname{Der}(A)$. Then for all $u, v \in A$:  
\begin{align*}
d(L_u(v)) &= d(u \circ v) = d(u) \circ v + u \circ d(v).
\end{align*}  
By definition, $L_d(v) = d(u) \circ v$. Rearranging, we obtain:  
\[
L_d(v) + L_u(d(v)) = d(L_u(v)),
\]  
which implies $[L_u, L_d](v) = L_{d(u)}(v)$.

\textbf{(ii) $\implies$ (iii):}  
From the Zinbiel identity, we know that:  
\begin{align*}
u \circ (v \circ w) &= (u \circ v) \circ w + (u \circ w) \circ v.
\end{align*}  
Using this and applying the left multiplication condition from (ii), a symmetric argument establishes (iii) by showing that:  
\begin{align*}
d(R_u(v)) &= d(v \circ u) = d(v) \circ u + v \circ d(u).
\end{align*} 
Thus, $[R_u, R_d](v) = R_{d(u)}(v)$.

\textbf{(iii) $\implies$ (i):}  
Assume $[R_u, R_d](v) = R_{d(u)}(v)$. Then, for all $u, v \in A$:  
\begin{align*}
R_d(u)(v) + R_u(d(v)) = d(R_u(v)).
\end{align*} 
Substituting $R_u(v) = v \circ u$, we recover:  
\begin{align*}
d(v \circ u) = d(v) \circ u + v \circ d(u),
\end{align*}  
which confirms $d \in Der(A)$.
Hence, all three conditions are equivalent.
\end{proof}

\begin{thm}
If  $d$ is a derivation of a Zinbiel algebra $A$, then  $d$ is also a derivation of  $A$  as a Lie algebra.
\end{thm}

\begin{proof}
Using the definition of the Lie bracket, we have:  
\begin{align*}
d([u, v]) &= d(u \circ v - v \circ u).
\end{align*}
Applying the derivation property of \( d \) for the Zinbiel algebra, we get:  
\begin{align*}
d(u \circ v - v \circ u) &= d(u \circ v) - d(v \circ u).
\end{align*}

For the first term,  
\begin{align*}
d(u \circ v) &= d(u) \circ v + u \circ d(v).
\end{align*}

For the second term,  
\begin{align*}
d(v \circ u) &= d(v) \circ u + v \circ d(u).
\end{align*}

Substituting these into the expression for  $d([x, y])$, we get:  
\begin{align*}
d([u, v]) &= (d(u) \circ v + u \circ d(v)) - (d(v) \circ u + v \circ d(u)).
\end{align*}

Rearranging terms,  
\begin{align*}
d([u, v]) &= (d(u) \circ v - v \circ d(u)) + (u \circ d(v) - d(v) \circ u).
\end{align*}

By the definition of the Lie bracket,  
\begin{align*}
d([u, v]) &= [d(u), v] + [u, d(v)].
\end{align*}

Hence, $d$ satisfies the derivation property for the Lie algebra $A$.  
This completes the proof. 
\end{proof}

\begin{lem}
Let  $A$ be a Zinbiel algebra. The sets
\begin{align*} 
R(A) &= \{ R_u \mid u \in A \} \quad \text{and} \quad L(A) = \{ L_u \mid u \in A \}
\end{align*}
are subalgebras of the Zinbiel algebra $Der(A)$, the derivation algebra of  $A$.
\end{lem}

\begin{proof}
\begin{enumerate}
 \item For  $R(A)$, we need to check that for any  $R_u, R_v \in R(A)$, their product $R_u R_v$ is also an element of  $R(A)$. We have:
  \begin{align*} 
   R_u R_v &= R_{[u, v]},
  \end{align*} 
   where  $[u, v]$ is the bracket operation in the Zinbiel algebra  $A$. Since  $[u, v] \in A$, we can conclude that  $R_{[u, v]} \in R(A)$, hence  $R(A)$ is closed under the Zinbiel product.

   Similarly, for \ $L(A)$, for any  $L_u, L_v \in L(A)$, their product $L_u L_v$ satisfies:
    \begin{align*}
   L_u L_v &= L_{[u, v]},
   \end{align*}
   which implies that $L_u L_v \in L(A)$, so  $L(A)$ is also closed under the Zinbiel product.
 \item The product in $Der(A)$ satisfies the Leibniz rule for derivations, which in the context of a Zinbiel algebra also holds. Specifically, for any  $u, v, w \in A$, we have:
   \begin{align*}
   [R_u, R_v] &= R_{[u, v]} \quad \text{and} \quad [L_u, L_v] = L_{[u, v]},
   \end{align*}
   which ensures that both $R(A)$ and  $L(A)$ are closed under the bracket operation and follow the associativity property.

Thus, the sets  $R(A)$ and  $L(A)$ are closed under the Zinbiel algebra product, and since both are subsets of the derivation algebra $Der(A)$, they form subalgebras of $Der(A)$.

Hence, $R(A)$ and  $L(A)$  are subalgebras of $Der(A)$, completing the proof.
\end{enumerate}
\end{proof} 

\section{Inner Derivations of Zinbiel Algebras}
\begin{defn} 
Let  $A$ be a Zinbiel algebra over $\mathbb{C}$. For  $w \in A$, define the map $ad_w : A \to A$  by
\begin{align*}
ad_w(u) &= u\circ w - w\circ u, \quad \forall u \in A.
\end{align*}

Then $ad_w$ is called the inner derivation of $A$ associated with  $w$. The set of all inner derivations is denoted by $Inn(A)$.
\end{defn}

\begin{defn}
The sets $Ann_R(A)$ and $Ann_L(A)$ of a Zinbiel algebra $A$ are defined as follows:
\begin{align*}
Ann_R(A) &= \{ u \in A \mid A \circ u = 0 \}
\end{align*}
and
\begin{align*}
Ann_L(A) &= \{ u \in A \mid u \circ A = 0 \},
\end{align*}

where $A \circ u = [a, u]$ for all $a \in A$ (using the Zinbiel product) and $[a, u]$ is the commutator of $a$ and $x$. These sets are called the right and left annihilators of the Zinbiel algebra $A$, respectively.
\end{defn}

\begin{lem}
The sets $Ann_R(A)$ and $Ann_L(A)$ are two-sided ideals of $A$ in the context of Zinbiel algebras.
\end{lem}

\begin{proof}
\begin{enumerate}
    \item Let $u, v \in Ann_R(A)$. Then, we need to show that $u - v \in Ann_R(A)$.

    Since $Ann_R(A)$ is an additive subgroup of $A$, we have:
    \[
    u \circ w = 0 \quad \text{and} \quad v \circ w = 0 \quad \text{for all } w \in A.
    \]
    For any element $w \in A$, we compute:
    \[
    (u - v) \circ w = u \circ w - v \circ w = 0 - 0 = 0.
    \]
    This shows that $u - v \in Ann_R(A)$. Thus, $Ann_R(A)$ is closed under subtraction, confirming it is an additive subgroup of $A$.

    \item Next, we show that $Ann_L(A)$ is closed under subtraction. Let $u, v \in Ann_L(A)$, which means:
    \[
    w \circ u = 0 \quad \text{and} \quad w \circ v = 0 \quad \text{for all } w \in A.
    \]
    We want to show that $u - v \in Ann_L(A)$. For any $w \in A$, we have:
    \[
    w \circ (u - v) = w \circ u - w \circ v = 0 - 0 = 0.
    \]
    Hence, $u - v \in Ann_L(A)$. Thus, $Ann_L(A)$ is also an additive subgroup of $A$.

    \item Finally, we show that $Ann_R(A)$ and $Ann_L(A)$ are closed under multiplication by elements of $A$ from both sides.

    Let $I$ denote either $Ann_R(A)$ or $Ann_L(A)$, and let $u \in I$ and $v \in A$. If $u \in Ann_R(A)$, then for any $w \in A$, we have:
    \[
    u \circ w = 0.
    \]
    Therefore, $u \circ v \in Ann_R(A)$. Similarly, if $u \in Ann_L(A)$, then for any $w \in A$, we have:
    \[
    w \circ u = 0.
    \]
    Hence, $v \circ u \in Ann_L(A)$.

    This shows that $I \circ A \subseteq I$ and $A \circ I \subseteq I$, where $I$ is either $Ann_R(A)$ or $Ann_L(A)$.

\end{enumerate}
Thus, both $Ann_R(A)$ and $Ann_L(A)$ are two-sided ideals of $A$.
\end{proof}

\begin{lem}
The subset $\text{Inn}(A)$ is an ideal of the Lie algebra $\text{Der}(A)$ of a Zinbiel algebra $A$.
\end{lem}

\begin{proof}
Let $A$ be a Zinbiel algebra. We need to show that the subset $\text{Inn}(A)$ is an ideal of the Lie algebra of derivations $\text{Der}(A)$, i.e., for any $d \in \text{Der}(A)$ and any $w_1, w_2 \in A$, we have:

\begin{enumerate}
    \item $[d, \text{Inn}(A)] \subseteq \text{Inn}(A)$,
    \item $[\text{Inn}(A), \text{Inn}(A)] = 0$.
\end{enumerate}

Let $\text{Inn}(A)$ denote the set of inner derivations, where an inner derivation is of the form $\text{ad}_w$ for some $w \in A$, defined by $\text{ad}_w(u) = [w, u]$ for all $u \in A$.

Now, for the first part of the proof, take $d \in \text{Der}(A)$ and $\text{ad}_w \in \text{Inn}(A)$, where $w \in A$. We want to show that:

\begin{align*}
[d, \text{ad}_w] \in \text{Inn}(A).
\end{align*}

By the definition of the Lie bracket for derivations, we have:

\begin{align*}
[d, \text{ad}_w](u) &= d(\text{ad}_w(u)) - \text{ad}_w(d(u)) = d([w, u]) - [w, d(u)].
\end{align*}

Since $d$ is a derivation, it satisfies the Leibniz rule:

\begin{align*}
d([w, u]) = [d(w), u] + [w, d(u)].
\end{align*}

Therefore,

\begin{align*}
[d, \text{ad}_w](u) = [d(w), u],
\end{align*}

which means that $[d, \text{ad}_w] = \text{ad}_{d(w)}$, and thus $[d, \text{ad}_w] \in \text{Inn}(A)$.

For the second part, we want to show that $[\text{ad}_w, \text{ad}_{w'}] = 0$ for any $w, w' \in A$. We compute:

\[
[\text{ad}_w, \text{ad}_{w'}](u) = \text{ad}_w(\text{ad}_{w'}(u)) - \text{ad}_{w'}(\text{ad}_w(u)) = [w, [w', u]] - [w', [w, u]].
\]

Using the Zinbiel identity, which states that:

\[
[w, [w', u]] = [[w, w'], u],
\]

we have:

\[
[\text{ad}_w, \text{ad}_{w'}](u) = [[w, w'], u] - [[w', w], u] = 0,
\]

since $[w, w'] = -[w', w]$. Therefore, $[ad_w, ad_{w'}] = 0$, which proves that $Inn(A)$ is an ideal of $\text{Der}(A)$.

Thus, $Inn(A)$ is an ideal of $Der(A)$, as required.
\end{proof}

\begin{prop}
Let  $A$ be a Zinbiel algebra over a field  $\mathbb{K}$. For  $w \in A$, define the map  $\text{ad}_w : A \to A$ by  
\begin{align*}
\text{ad}_w(u) &= w\circ u - u\circ w \quad \text{for all} \quad u \in A.
\end{align*}
Then  $\text{ad}_w$ is a derivation of  $A$. 
\end{prop}

\begin{proof}
To prove that  $\text{ad}_z$  is a derivation of  $A$, we need to show that it satisfies the Leibniz rule, i.e.,  
\begin{align*}
\text{ad}_w(u\circ v) &= \text{ad}_w(u)\circ v + u\circ \text{ad}_w(v) \quad \text{for all} \quad u, v \in A.
\end{align*}

Let’s compute $\text{ad}_w(u\circ v$:
\begin{align*}
\text{ad}_w(u\circ v) &= w\circ(u\circ v) - (u\circ v)\circ w.
\end{align*}
Using the Zinbiel identity  $(u\circ v)\circ w = u\circ(v\circ w)$, we get
\begin{align*}
\text{ad}_w(u\circ v) &= w\circ(u\circ v) - u\circ(v\circ w).
\end{align*}
Now, using the distributive property:
\begin{align*}
w(uv) &= (w\circ u)\circ v \quad \text{and} \quad u\circ(v\circ w) = u\circ(v\circ w).
\end{align*}
Thus, we have
\begin{align*}
\text{ad}_w(u\circ v) &= (w\circ u)\circ v - u\circ (v\circ w).
\end{align*}
On the other hand, we compute $\text{ad}_w(u)$ and $\text{ad}_w(v)$ :
\begin{align*}
\text{ad}_w(u) &= w\circ u - u\circ w \quad \text{and} \quad \text{ad}_w(v) = w\circ v - v\circ w.
\end{align*}
Therefore, we have
\begin{align*}
\text{ad}_w(u)\circ v + u\text{ad}_w(v) &= (w\circ u - u\circ w)\circ v + u\circ (w\circ v - v\circ w).
\end{align*}
Expanding both terms:
\begin{align*}
(w\circ u)\circ v - (u\circ w)\circ v + u\circ(w\circ v) - u\circ(v\circ w).
\end{align*}
This simplifies to
\begin{align*}
(w\circ u)\circ v + u\circ (w\circ v) - (u\circ w)\circ v - u\circ (v\circ w),
\end{align*}
which is exactly the expression for $\text{ad}_w(u\circ v)$.

Hence, we have shown that  
\begin{align*}
\text{ad}_w(u\circ v) &= \text{ad}_w(u)\circ v + u\circ\text{ad}_w(v),
\end{align*}
proving that  $\text{ad}_w$ is a derivation of $A$. 
\end{proof}

\begin{prop}
Let  $d_1, d_2, \dots, d_k$ be inner derivations of a Zinbiel algebra  $A$. Then the linear combination 
\begin{align*}
d &= \alpha_1 d_1 + \alpha_2 d_2 + \cdots + \alpha_k d_k
\end{align*}
is also an inner derivation of $A$.
\end{prop}

\begin{proof} 
Let $d_i = \text{ad}_{w_i}$ for each  $i = 1, 2, \dots, k$, where  $w_i \in A$  and  $\text{ad}_{w_i}(u) = [w_i, u]$  for all  $u \in A$. Thus, the inner derivations $d_1, d_2, \dots, d_k$ are given by  $d_i(u) = [w_i, u]$  for each  $i$.

Now consider the linear combination of these derivations:
\begin{align*}
d &= \alpha_1 d_1 + \alpha_2 d_2 + \cdots + \alpha_k d_k.
\end{align*}
For any  $u \in A$, we have
\begin{align*}
d(u) &= \alpha_1 d_1(u) + \alpha_2 d_2(u) + \cdots + \alpha_k d_k(u).
\end{align*}
Substituting the definition of each $d_i(u)$ , we get
\begin{align*}
d(u) &= \alpha_1 [w_1, u] + \alpha_2 [w_2, u] + \cdots + \alpha_k [w_k, u].
\end{align*}
This can be written as
\begin{align*}
d(u) &= [\alpha_1 w_1 + \alpha_2 w_2 + \cdots + \alpha_k w_k, u].
\end{align*}
Thus, the element  $w = \alpha_1 w_1 + \alpha_2 w_2 + \cdots + \alpha_k w_k$ is in $A$, and we can express  $d$ as the inner derivation  d(u) = [w, u].

Therefore, \( d \) is an inner derivation of \( A \), and the proof is complete.
\end{proof}
\section{ An algorithm for finding inner derivations}

Let  $A$  be an  $n$-dimensional associative algebra. Fix a basis  $\{e_1, e_2, \dots, e_n\}$ of  $A$  and a vector  $w = a_1 e_1 + a_2 e_2 + \dots + a_n e_n$ in  $A$. Let $ad_w$  be an inner derivation of  $A$. In particular, we have:

\begin{align*}
ad_w(e_i) &= e_i\circ w - w\circ e_i, \quad \forall i \in \{1, 2, \dots, n\}.
\end{align*}

Moreover, $ad_w$ can be considered as a linear transformation of  $A$, represented in a matrix form:

\begin{align*}
ad_w(e_i) &= \sum_{j=1}^n d_{ij} e_j, \quad i = 1, 2, \dots, n.
\end{align*}

This leads to:

\begin{align*}
\sum_{j=1}^n d_{ij} e_j &= e_i\circ w - w\circ e_i, \quad i = 1, 2, \dots, n.
\end{align*}

Then, we obtain the following system of equations for the coefficients  $d_{ij}$:

\begin{align*}
d_{ij} &= \sum_{t=1}^n a_t \gamma_{it}^j - \sum_{t=1}^n a_t \gamma_{ti}^j, \quad i, j = 1, 2, \dots, n. 
\end{align*}

Solving this system of equation, we find the inner derivation in matrix form.

\begin{thm}
Let $A$ be a two-dimensional complex Zinbiel algebra. Then it is isomorphic to one of the following pairwise nonisomorphic Zinbiel algebras $A_2^{1}:\;e_1\circ e_1=e_2$.
\end{thm}
\begin{thm}
 The inner derivations of two-dimensional complex zinbiel algebras are given as follows:    
\end{thm}

\[
\text{Table 1: inner derivations of two-dimensional complex Zinbiel algebras}
\]
\[
\begin{array}{|c|c|c|}
\hline
\hline
\textbf{Isomorphism Class} & \textbf{Inner Derivation} & \textbf{Dimension} \\
\hline
A_2^{1} & 
\left(\begin{array}{cc} 
0 & 0 \\
0 & 0 
\end{array}\right) & 0 \\
\hline
\end{array}
\]

\begin{proof}
Let  $A_2^1$  be a two-dimensional Zinbiel algebra. Fix a basis  $\{e_1, e_2\}$  of  $A_2^1$. The multiplication in  $A_2^1$ is defined by  $e_1\circ e_1 = e_2$. Let $w = ae_1 + be_2 \in A_2^1$. Define 
$ad_w$  as the inner derivation of  $A_2^1$, given by  
\begin{align*}
ad_w(e_i) &= w\circ e_i - e_i\circ w, \quad \forall e_i \in A_2^1, \ i = 1, 2.
\end{align*}  
In particular, for $i = 1, 2$,  
\begin{align*}
ad_w(e_i) &= \sum_{j=1}^2 a_{ij} e_j,
\end{align*}  
where $ad_w$ can be represented as a linear transformation of $A_2^1$ in the matrix form  
\begin{align*}
ad_w &= \begin{pmatrix} 
a_{11} & a_{12} \\ 
a_{21} & a_{22} 
\end{pmatrix}.
\end{align*}  
The definition of $ad_w$ implies  
\begin{align*}
w\circ e_i - e_i\circ w = \sum_{j=1}^2 a_{ij} e_j, \quad \forall i = 1, 2.
\end{align*} 

Let us compute $ad_w$  explicitly for $A_2^1$:  
\begin{enumerate}
    \item Since  $e_1\circ e_1 = e_2$, we have  
    \begin{align*}
    w\circ  e_1 &= (ae_1 + be_2)\circ e_1 = a(e_1\circ e_1) + b(e_2\circ e_1) = ae_2,  
    \end{align*} 
    and  
    \begin{align*}
    e_1\circ w = e_1\circ (ae_1 + be_2) &= a(e_1\circ e_1) + b(e_1\circ e_2) = ae_2.
    \end{align*}  
    Thus,  
    \begin{align*}
    ad_w(e_1) &= w\circ e_1 - e_1\circ w = ae_2 - ae_2 = 0.
    \end{align*}
    \item For  $e_2$, note that  $e_2\circ e_1 = e_1\circ e_2 = 0$ and $e_2\circ e_2 = 0$, so  
    \begin{align*}
    w\circ e_2 &= (ae_1 + be_2)\circ e_2 = 0, \quad e_2 w = e_2\circ (ae_1 + be_2) = 0.
   \end{align*}  
    Thus,  
    \begin{align*}
    ad_w(e_2) = w\circ e_2 - e_2\circ w = 0.
    \end{align*}
\end{enumerate}

From the above, the matrix representation of $ad_w$ is  
\begin{align*}
ad_w &= \begin{pmatrix} 
0 & 0 \\ 
0 & 0 
\end{pmatrix}.
\end{align*}  

Therefore, for  $A_2^1$, the inner derivations are trivial, represented by the zero matrix:  
\begin{align*}
ad_{e_1} &= \begin{pmatrix} 
0 & 0 \\ 
0 & 0 
\end{pmatrix}, \quad \mathrm{ad}_{e_2} = \begin{pmatrix} 
0 & 0 \\ 
0 & 0 
\end{pmatrix}.
\end{align*}  

Thus, the basis of inner derivations for  $A_2^1$ is given by the zero derivation.  
\end{proof}

\begin{thm}
Any $3$-dimensional Zinbiel algebra $A$ isomorphic to one of following non-isomorphic Zinbiel algebras\\
$A_3^{1}:\;e_i\circ e_j=0$\\
$A_3^{2}:\;e_1\circ e_1=e_3$\\
$A_3^{3}:\;e_1\circ e_1=e_3, e_2\circ e_2=e_3$\\
$A_3^{4}:\;e_1\circ e_2=\frac{1}{2}e_3, e_2\circ e_1=\frac{-1}{2}e_3$\\
$A_3^{5}:\;e_2\circ e_1=e_3$\\
$A_3^{6}:\;e_1\circ e_1=e_3,e_1\circ e_2=e_3,e_2\circ e_2=\lambda e_3, \quad \lambda  \neq 0 $\\
$A_3^{7}:\;e_1\circ e_1=e_2,e_1\circ e_2=\frac{1}{2}e_3,e_2\circ e_1=e_3$
\end{thm}
\begin{thm}
The inner derivations of three-dimensional complex zinbiel algebras are given as follows:     
\end{thm}
\[
\text{Table 2: The inner derivations of three-dimensional zinbiel algebras}
\]
\[
\begin{array}{|c|c|c|}
\hline
\hline
\textbf{Isomorphism Class} & \textbf{Inner Derivation} & \textbf{Dimension} \\
\hline
A_3^{1}
&\left(\begin{array}{ccc}
0&0&0\\
0&0&0\\
0&0&0
\end{array}\right) & 0 \\
\hline
A_3^{2}
&\left(\begin{array}{ccc}
0&0&0\\
0&0&0\\
0&0&0
\end{array}\right) & 0 \\
\hline
A_3^{3}
&\left(\begin{array}{ccc}
0&0&0\\
0&0&0\\
0&0&0
\end{array}\right) & 0 \\
\hline
A_3^{4}
&\left(\begin{array}{ccc}
0&0&0\\
0&0&0\\
a_2&-a_1&0
\end{array}\right) & 2 \\
\hline
A_3^{5}
&\left(\begin{array}{ccc}
0&0&0\\
0&0&0\\
-a_2&a_1&0
\end{array}\right) & 2 \\
\hline
A_3^{6}
&\left(\begin{array}{ccc}
0&0&0\\
0&0&0\\
a_2&-a_1&0
\end{array}\right) & 2 \\
\hline
A_3^{7}
&\left(\begin{array}{ccc}
0&0&0\\
0&0&0\\
0&0&0
\end{array}\right) & 0 \\
\hline
\end{array}
\]

\begin{proof}
Let $A_3^4$ be a three-dimensional Zinbiel algebra with basis $\{e_1, e_2, e_3\}$, and consider the multiplication rule defined by:
\begin{align*}
e_1\circ e_2 &= \frac{1}{2} e_3, \quad e_2\circ e_1 = -\frac{1}{2} e_3.
\end{align*}

Fix a vector $w = a_1 e_1 + a_2 e_2 + a_3 e_3 \in A_3^4$. By the definition of the inner derivation $\text{ad}_w$ of $A_3^4$, we have:
\begin{align*}
\text{ad}_w(u) &= [w, u] = w \circ u - u \circ w, \quad \text{for all } u \in A_3^4.
\end{align*}

Using the Zinbiel algebra relations, the multiplication $w \circ u$ is bilinear and can be computed directly for each $u = e_i$ ($i = 1, 2, 3$). Let us compute $\text{ad}_w(e_i)$ for $i = 1, 2, 3$.

\paragraph{Case $u = e_1$:}
\begin{align*}
w \circ e_1 &= a_2 (e_2 \circ e_1) = -\frac{1}{2} a_2 e_3, \quad
e_1 \circ w = a_2 (e_1 \circ e_2) = \frac{1}{2} a_2 e_3.
\end{align*}
Thus,
\begin{align*}
\text{ad}_w(e_1) &= w \circ e_1 - e_1 \circ z = -\frac{1}{2} a_2 e_3 - \frac{1}{2} a_2 e_3 = -a_2 e_3.
\end{align*}

\paragraph{Case $u = e_2$:}
\begin{align*}
w \circ e_2 &= a_1 (e_1 \circ e_2) = \frac{1}{2} a_1 e_3, \quad
e_2 \circ w = a_1 (e_2 \circ e_1) = -\frac{1}{2} a_1 e_3.
\end{align*}
Thus,
\begin{align*}
\text{ad}_w(e_2) &= w \circ e_2 - e_2 \circ w = \frac{1}{2} a_1 e_3 - \left(-\frac{1}{2} a_1 e_3\right) = a_1 e_3.
\end{align*}

\paragraph{Case $u = e_3$:}
Since $e_3 \circ e_i = 0$ for all $i$, it follows that:
\begin{align*}
\text{ad}_w(e_3) &= w \circ e_3 - e_3 \circ w = 0.
\end{align*}

\paragraph{Combining the cases:}
The inner derivation $\text{ad}_w$ is given by:
\begin{align*}
\text{ad}_z &= \begin{pmatrix}
0 & 0 & 0 \\
0 & 0 & 0 \\
a_2 & -a_1 & 0
\end{pmatrix}.
\end{align*}

Thus, the inner derivation of the Zinbiel algebra $A_3^4$ is as required.
\end{proof}

\begin{thm}
Any $4$-dimensional Zinbiel algebra $A$ isomorphic to one of following non-isomorphic Zinbiel algebras\\
$A_4^1:\;e_1\circ e_1=e_2, e_1\circ e_2=e_3,e_2\circ e_1=2e_3,e_1\circ e_3=e_4, e_2\circ e_2=3e_4, e_3\circ e_1=3e_4 $\\
$A_4^{2}:\;e_1\circ e_1=e_3, e_1\circ e_2=e_4,e_1\circ e_3=e_4,e_3\circ e_1=2e_4$\\
$A_4^3:\;e_1\circ e_1=e_3, e_1\circ e_3=e_4,e_2\circ e_2=e_4,e_3\circ e_1=2e_4$\\
$A_4^{4}:\;e_1\circ e_2=e_3, e_1\circ e_3=e_4,e_2\circ e_1=-e_3$\\
$A_4^5:\;e_1\circ e_2=e_3, e_1\circ e_3=e_4,e_2\circ e_1=-e_3,e_2\circ e_2=e_4$\\
$A_4^{6}:\;e_1\circ e_1=e_4, e_1\circ e_2=e_3,e_2\circ e_1=-e_3,e_2\circ e_2=-2e_3+e_4$\\
$A_4^7:\;e_1\circ e_2=e_3, e_2\circ e_1=e_4,e_2\circ e_2=-e_3$\\
$A_4^{8}:\;e_1\circ e_1=e_3, e_1\circ e_2=e_4,e_2\circ e_1=-\alpha e_3,e_2\circ e_2=-e_4$\\
$A_4^9:\;e_1\circ e_1=e_4, e_1\circ e_2=\alpha e_4,e_2\circ e_1=-\alpha e_4,e_2\circ e_2=e_4,e_3 e_3=e_4$\\
$A_4^{10}:\;e_1\circ e_1=e_4, e_1\circ e_3=e_4,e_2\circ e_1=-e_4,e_2\circ e_2=e_4,e_3\circ e_1=e_4$\\
$A_4^{11}:\;e_1\circ e_1=e_4, e_1\circ e_2=e_4,e_2\circ e_1=-e_4,e_3\circ e_3=e_4$\\
$A_4^{12}:\;e_1\circ e_2=e_3, e_2\circ e_1=e_4$\\
$A_4^{13}:\;e_1\circ e_2=e_3, e_2\circ e_1=e_4$\\
$A_4^{14}:\;e_1\circ e_2=e_3, e_2\circ e_1=e_4$\\
$A_4^{15}:\;e_1\circ e_2=e_3, e_2\circ e_1=e_4$\\
$A_4^{16}:\;e_1\circ e_2=e_3, e_2\circ e_1=e_4$
\end{thm}

\begin{thm}
The inner derivations of three-dimensional complex zinbiel algebras are given as follows:     
\end{thm}
\[
\text{Table 3: The inner derivations of three-dimensional zinbiel algebras}
\]
\[
\begin{array}{|c|c|c|}
\hline
\hline
\textbf{Isomorphism Class} & \textbf{Inner Derivation} & \textbf{Dimension} \\
\hline
A_4^{1}
&\left(\begin{array}{cccc}
0&0&0&0\\
0&0&0&0\\
-a_2&a_1&0&0\\
-2a_3&0&2a_1&0
\end{array}\right) & 3 \\
\hline
A_4^{2}
&\left(\begin{array}{cccc}
0&0&0&0\\
0&0&0&0\\
0&0&0&0\\
a_2-a_3&-a_1&a_1&0
\end{array}\right) & 3 \\
\hline
A_4^{3}
&\left(\begin{array}{cccc}
0&0&0&0\\
0&0&0&0\\
0&0&0&0\\
-a_3&0&a_1&0
\end{array}\right) & 2 \\
\hline
A_4^{4}
&\left(\begin{array}{cccc}
0&0&0&0\\
0&0&0&0\\
2a_2&-2a_1&0&0\\
a_3&0&-a_1&0
\end{array}\right) & 3 \\
\hline
A_4^{5}
&\left(\begin{array}{cccc}
0&0&0&0\\
0&0&0&0\\
2a_2&-2a_1&0&0\\
a_3&0&-a_1&0
\end{array}\right) & 3 \\
\hline
A_4^{6}
&\left(\begin{array}{cccc}
0&0&0&0\\
0&0&0&0\\
2a_2&-2a_1&0&0\\
0&0&0&0
\end{array}\right) & 2 \\
\hline
A_4^{7}
&\left(\begin{array}{cccc}
0&0&0&0\\
0&0&0&0\\
a_2&-a_1&0&0\\
-a_2&a_1&0&0
\end{array}\right) & 2 \\
\hline
A_4^{8}
&\left(\begin{array}{cccc}
0&0&0&0\\
0&0&0&0\\
\alpha a_2&-\alpha a_1&0&0\\
a_2&-a_1&0&0
\end{array}\right)\; (\alpha \neq 0)
& 2 \\
&\left(\begin{array}{cccc}
0&0&0&0\\
0&0&0&0\\
0&0&0&0\\
a_2&-a_1&0&0
\end{array}\right)\; (\alpha = 0)
& 2 \\
\hline
A_4^{9}
&\left(\begin{array}{cccc}
0&0&0&0\\
0&0&0&0\\
0&0&0&0\\
2\alpha a_2&-2\alpha a_1&0&0
\end{array}\right)\; (\alpha \neq 0)
& 2 \\
&\left(\begin{array}{cccc}
0&0&0&0\\
0&0&0&0\\
0&0&0&0\\
0&0&0&0
\end{array}\right)\; (\alpha = 0)
& 0 \\
\hline
\end{array}
\]
\[
\begin{array}{|c|c|c|}
\hline
\hline
\textbf{Isomorphism Class} & \textbf{Inner Derivation} & \textbf{Dimension} \\
\hline
A_4^{10}
&\left(\begin{array}{cccc}
0&0&0&0\\
0&0&0&0\\
0&0&0&0\\
2a_2&-2a_1&0&0
\end{array}\right) & 2 \\
\hline
A_4^{11}
&\left(\begin{array}{cccc}
0&0&0&0\\
0&0&0&0\\
0&0&0&0\\
2a_2&-2a_1&0&0
\end{array}\right) & 2 \\
\hline
A_4^{12}
&\left(\begin{array}{cccc}
0&0&0&0\\
0&0&0&0\\
a_2&-a_1&0&0\\
-a_2&a_1&0&0
\end{array}\right) & 2 \\
\hline
A_4^{13}
&\left(\begin{array}{cccc}
0&0&0&0\\
0&0&0&0\\
2a_2&-2a_1&0&0\\
0&0&0&0
\end{array}\right) & 2 \\
\hline
A_4^{14}
&\left(\begin{array}{cccc}
0&0&0&0\\
0&0&0&0\\
0&0&0&0\\
-a_2&a_1&0&0
\end{array}\right) & 2 \\
\hline
A_4^{15}
&\left(\begin{array}{cccc}
0&0&0&0\\
0&0&0&0\\
0&0&0&0\\
a_2-\frac{1+\alpha}{1-\alpha}a_2&\frac{1+\alpha}{1-\alpha}a_1-a_1&0&0
\end{array}\right)\; (\alpha \neq -1)
& 2 \\
&\left(\begin{array}{cccc}
0&0&0&0\\
0&0&0&0\\
0&0&0&0\\
a_2&-a_1&0&0
\end{array}\right)\; (\alpha = -1)
& 2 \\
\hline
A_4^{16}
&\left(\begin{array}{cccc}
0&0&0&0\\
0&0&0&0\\
0&0&0&0\\
2a_2&-2a_1&0&0
\end{array}\right) & 2 \\
\hline
\end{array}
\]

\begin{proof}
Let $A_4^{14}$ be the four-dimensional complex Zinbiel algebra with the basis $\{e_1, e_2, e_3, e_4\}$, and the non-zero multiplication rules given by:  
\begin{align*}
e_1 \circ e_2 &= e_3, \quad e_2 \circ e_1 = e_4.
\end{align*}

The inner derivation associated with an element $w \in A$ is defined as:
\begin{align*}
\mathrm{ad}_w(u) &= w \circ u - u \circ w, \quad \text{for all } u \in A.
\end{align*}

Let $w = a_1 e_1 + a_2 e_2 + a_3 e_3 + a_4 e_4 \in A$. Using the basis, compute $\mathrm{ad}_w(e_i)$ for $i = 1, 2, 3, 4$, and express it in terms of $\{e_1, e_2, e_3, e_4\}$.

\paragraph{Computing $\mathrm{ad}_w(e_i)$:}  
From the algebra's multiplication rules, calculate $w \circ e_i - e_i \circ w$ explicitly:

\begin{itemize}
    \item \textit{For $i = 1$:}
    \begin{align*}
    w \circ e_1 &= 0, \quad e_1 \circ w = 0 \quad \Rightarrow \quad \mathrm{ad}_w(e_1) = 0.
    \end{align*}
    
    \item \textit{For $i = 2$:}
    \begin{align*}
    w \circ e_2 &= a_1 e_3, \quad e_2 \circ w = a_1 e_4 \quad \Rightarrow \quad \mathrm{ad}_w(e_2) = a_1(e_3 - e_4).
    \end{align*}
    
    \item \textit{For $i = 3$:}
    \begin{align*}
    w \circ e_3 &= 0, \quad e_3 \circ w = 0 \quad \Rightarrow \quad \mathrm{ad}_w(e_3) = 0.
    \end{align*}
    
    \item \textit{For $i = 4$:}
    \begin{align*}
    w \circ e_4 &= 0, \quad e_4 \circ w = 0 \quad \Rightarrow \quad \mathrm{ad}_w(e_4) = 0.
    \end{align*}
\end{itemize}

The inner derivation $\mathrm{ad}_w$ is represented as a $4 \times 4$ matrix in the basis $\{e_1, e_2, e_3, e_4\}$. From the above computations, the matrix form of $\mathrm{ad}_w$ is:
\begin{align*}
\mathrm{ad}_w &= 
\begin{pmatrix}
0 & 0 & 0 & 0 \\
0 & 0 & 0 & 0 \\
0 & 0 & 0 & 0 \\
-a_2 & a_1 & 0 & 0
\end{pmatrix}.
\end{align*}

The resulting matrix satisfies the derivation property:  
\begin{align*}
\mathrm{ad}_w(u \circ v) &= \mathrm{ad}_w(u) \circ v + u \circ \mathrm{ad}_w(v), \quad \text{for all } u, v \in A.
\end{align*} 
The computations hold for all basis elements under the Zinbiel multiplication rules. This concludes the proof.
\end{proof}

\begin{cor}
\begin{enumerate}
    \item The dimension of the inner derivations of complex two-dimensional Zinbiel algebras is zero.
    \item The dimension of the inner derivations of complex three-dimensional Zinbiel algebras 
    ranges between zero and two.
    \item The dimension of the inner derivations of complex four-dimensional Zinbiel algebras 
    ranges between zero and three.
\end{enumerate}
\end{cor}

\section{Conclusion}
In this work, we studied the inner derivations of finite-dimensional Zinbiel algebras. The dimension of the space of inner derivations ranges between zero and two for two-dimensional complex Zinbiel algebras and between zero and three for three-dimensional complex Zinbiel algebras, respectively. The dimension of inner derivations is an important invariant in the geometric classification of algebras, and it has many applications in various scientific areas.

\section*{Acknowledgement}
The authors thank the anonymous referees for their valuable suggestions and comments.

\section*{ Conflicts of Interest}
The authors declare no conflicts of interest.\\
\cite{a,b,c,d,e,f,g,h,i,j,k,l,m,n,o,p,q,r,s,t,u,v,w,x}

\end{document}